\documentclass[english]{amsart}
\usepackage[T1]{fontenc}
\usepackage[latin9]{inputenc}
\usepackage{amsmath}
\usepackage{amsthm}
\usepackage{amssymb}

\makeatletter
\theoremstyle{plain}
\newtheorem{thm}{\protect\theoremname}
  \theoremstyle{remark}
  \newtheorem{rem}[thm]{\protect\remarkname}
  \theoremstyle{plain}
  \newtheorem{lem}[thm]{\protect\lemmaname}
  \theoremstyle{plain}
  \newtheorem{prop}[thm]{\protect\propositionname}

\DeclareMathOperator{\cat}{cat}

\usepackage{hyperref}

\makeatother

\usepackage{babel}
  \providecommand{\lemmaname}{Lemma}
  \providecommand{\propositionname}{Proposition}
  \providecommand{\remarkname}{Remark}
\providecommand{\theoremname}{Theorem}

\begin{document}

\title[Positive solutions for Schroedinger Maxwell systems]{Positive solutions for double singularly perturbed Schroedinger Maxwell systems}
\thanks{The authors were  supported by Gruppo Nazionale per l'Analisi Matematica, la Probabilit\`{a} e le loro 
Applicazioni (GNAMPA) of Istituto Nazionale di Alta Matematica (INdAM)}

\author{Marco Ghimenti}
\address{M. Ghimenti, \newline Dipartimento di Matematica Universit\`a di Pisa
Largo B. Pontecorvo 5, 56100 Pisa, Italy}
\email{marco.ghimenti@unipi.it}

\author{Anna Maria Micheletti}
\address{A. M. Micheletti, \newline Dipartimento di Matematica Universit\`a di Pisa
Largo B. Pontecorvo 5, 56100 Pisa, Italy}
\email{a.micheletti@dma.unipi.it.}

\begin{abstract}
We show that the number of solutions of a double singularly perturbed
Schroedinger Maxwell system on a smooth bounded domain $\Omega\subset\mathbb{R}^{3}$
depends on the topological properties of the domain. In particular
if $\Omega$ is non contractible we obtain $\cat(\Omega)+1$ positive
solutions. The result is obtained via Lusternik Schnirelmann category
theory
\end{abstract}



\maketitle

\section{Introduction}

Given real numbers $q>0$, $\omega>0$ we consider the following Schroedinger
Maxwell stationary system on a smooth bounded domain $\Omega\subset\mathbb{R}^{3}$.
\begin{equation}
\left\{ \begin{array}{cc}
-\varepsilon^{2}\Delta u+u+\omega uv=|u|^{p-2}u & \text{ in }\Omega\\
-\varepsilon^{2}\Delta v=qu^{2} & \text{ in }\Omega\\
u,v=0\text{ on }\partial\Omega,\ u>0\text{ in }\Omega
\end{array}\right.\label{eq:sms}
\end{equation}

We want to prove that when the parameter $\varepsilon$ is sufficiently
small, there are many low energy solution of (\ref{eq:sms}). In particular
the number of solutions of (\ref{eq:sms}) is related to the topology
of the bounded set $\Omega$. 

Schroedinger Maxwell systems recently received considerable attention
from the mathematical community \cite{AP,BF,D,DW1,DW2,He,K,R}.
For a special case of stationary Schroedinger Maxwell type systems,
namely when the system is set in $\mathbb{R}^{3}$, we have an esplicit
expression for the function $v$
\[
v(u)=\frac{q}{4\pi}\int_{\mathbb{R}^{3}}\frac{u^{2}(y)}{|x-y|}dy,
\]
and the system is reduced to the following single nonlinear equation:
\[
-\Delta u+u+\frac{\omega q}{4\pi}\left(u^{2}\ast\frac{1}{|x|}\right)u=|u|^{p-2}u.
\]
This equation is also referred as Schroedinger-Poisson-Slater equation
and arises in the Slater approximation in the Hartree-Fock model (see
\cite{AR,BJL,He,IR,IV,PS,R10,S} and the reference therein). 

Coming back to the initial sistem, the singular perturbation of the
first equation is widely analysed in literature - we cite, for instance,
\cite{BDM,DW1,DW2,GM1} and the reference therein. More recently
the mathematical community moved to consider the double perturbed
problem \cite{GM15b,He,HZ,WTXZ,Y}, that is when the singular parameter
appears in both equations. In \cite{DS} the authors study the evolution
of a Schroedinger-Newton system, and it turns out that the double
perturbation is needed in order to prove the dynamics of solitary
waves when the parameters tend to zero.

Concerning existence of solutions, He \cite{He} studies the following
problem 
\[
\left\{ \begin{array}{cc}
-\varepsilon^{2}\Delta u+V(x)u+uv=f(u) & \text{ in }\mathbb{R}^{3}\\
-\varepsilon^{2}\Delta v=qu^{2} & \text{ in }\mathbb{R}^{3}\\
u>0
\end{array}\right.
\]
where $f$ is a subcritical nonlinearity and $V$ is a suitable potential,
while Yang \cite{Y} is interested to the system with critical nonlinearity
\[
\left\{ \begin{array}{cc}
-\varepsilon^{2}\Delta u+V(x)u+K(x)uv=P(x)g(u)+Q(x)|u|^{4}u & \text{ in }\mathbb{R}^{3}\\
-\varepsilon^{2}\Delta v=qu^{2} & \text{ in }\mathbb{R}^{3}\\
u>0
\end{array}\right.
\]
where $V,K,P,Q$ are suitable nonhomogeneous potentials. In both cases
the existence and multiplicity of solution is given by the properties
of the functions $V,K,P,Q$. The role of the topological properties
of the domain on the existence of solution is studied in \cite{GM15b},
in which a double perturbed nonlinear system is solved on a Riemannian
manifold without boundary. In all these papers a key role is played
by the limit problem of the type
\[
\left\{ \begin{array}{cc}
-\Delta u+u+\omega uv=|u|^{p-2}u & \text{ in }\mathbb{R}^{3}\\
-\Delta v=qu^{2} & \text{ in }\mathbb{R}^{3}\\
u>0\text{ in }\mathbb{R}^{3}
\end{array}\right.
\]
and the ground state solutions of this problem will provide a model
profile to construct solution for the original problem. 

The main difference when the domain has a boundary comes out when
looking for the limit problem. In fact, blowing down around a point
on the boundary $q\in\partial\Omega$ leads to a problem settled in
the half space. The main features of the limit problems in $\mathbb{R}^{3}$
and in the half space are recalled in Section \ref{subsec:The-limit-problem}
and will be crucial for our result.

Our main results is the following.
\begin{thm}
\label{thm:1}Let $4<p<6$. For $\varepsilon$ small enough there
exist at least $\cat(\Omega)$ low energy positive solutions of (\ref{eq:sms})
. Moreover if $\Omega$ is non contractible there is another positive
solution with higher energy.
\end{thm}
We recall that, given $X$ a topological space and a closed subset
$A\subset X$, we say that $A$ has category $k$ relative to $X$
($\cat_{X}A=k$) if $A$ is covered by $k$ closed sets $A_{j}$,
$j=1,\dots,k$, which are contractible in $X$, and $k$ is the minimum
integer with this property. We simply denote $\cat X=\cat_{X}X$. 
\begin{rem}
To prove our result, we construct two continuous operator, one - the
map $\Phi_{\varepsilon}$ - from the bounded set $\Omega$ to the
subset of low energy solution in $H_{0}^{1}(\Omega)$ and the second
-the barycenter map - from the subset of low energy solution in $H_{0}^{1}(\Omega)$
to the set $\Omega$, so that the composition is homotopically equivalent
to the identity map. A scheme of the proof of Theorem \ref{thm:1}
is given in Section \ref{sec:proof}. The main point of this paper
is contained in Section \ref{sec:Conc}: in fact, to define the barycenter
map we have to prove that a low energy function does not concentrate
near the boundary. This property relies on an adaptiation of the interesting
result by Esteban and Lions \cite{EL}, which state that a large class
of nonlinear elliptic partial differential equation with Dirichlet
boundary condition in the half space admits only the trivial solution. 
\end{rem}
\begin{rem}
It is interesting to ask what happens if we consider the case of Schroedinger
- Newton type equation, that is the case of equation (\ref{eq:sms})
with $\omega<0$. In this case the limit problem could be reduced,
by a simple change of variables, to a variational system. In this
case, the result of \cite{EL} applies directly. The other main difference
is that we are not able to prove the concentration result (i.e. Lemma
\ref{lem:gamma}) for the positive function $u^{+}$, but only for
$u$. Thus, it is not possible to state the final Theorem for positive
solutions. However, one can obtain a result of the type ``Problem
(\ref{eq:sms}) adimts at least $\cat(\Omega)/2$ pairs of low energy
solutions $(u,-u)$''.
\end{rem}
\begin{rem}
It is not known whether if a least energy solution of the Schroedinger
Maxwell system, or of the Schroedinger Poisson Slater equation is
unique or at least non degenerate (see \cite{IR,MV}). We want to
stress that the method we employ does not require any nondegeneracy
assumption: any ground state of the limit problem works perfectly
in the same way. A backing effect of the lack of non denegeracy, is
the obstruction to prove a multiplicity result by using finite dimensional
reduction, as the well known Liapunov-Schmidt procedure. Also, the
lack of uniqueness of ground state for the limit problem (\ref{eq:PL})
is an obstruction to describe the asymptotic profile of the low energy
solutions when $\varepsilon\rightarrow0$. For example, applying the
same tecnique we use in this paper, one can prove that any solution
of (\ref{eq:sms}) with sufficiently low energy has a maximum point
$P_{\varepsilon}$ with $\frac{d(P_{\varepsilon},\partial\Omega)}{\varepsilon}\rightarrow\infty$
as $\varepsilon\rightarrow0$, and that if $u_{\varepsilon}$ has
two maximum points $P_{\varepsilon}$ and $Q_{\varepsilon}$ then
$P_{\varepsilon}$ and $Q_{\varepsilon}$ collide while $\varepsilon\rightarrow0$.
Unfortunately, without any a priori knowledge of the limiting profiles,
we can not prove that the maximum point is indeed unique, and to provide
a precise description of the profile around $P_{\varepsilon}$. 
\end{rem}

\section{Preliminary results}

We endow $H_{0}^{1}(\Omega)$ and $L^{p}(\Omega)$ with the following
norms equivalent to the standard ones 
\begin{eqnarray*}
\|u\|_{\varepsilon}^{2}=\frac{1}{\varepsilon^{3}}\int_{\Omega}\left(\varepsilon^{2}|\nabla u|^{2}+u^{2}\right)dx &  & |u|_{\varepsilon,p}^{p}=\frac{1}{\varepsilon^{3}}\int_{\Omega}|u|^{p}dx\\
\|u\|_{H_{0}^{1}}^{2}=\int_{\Omega}|\nabla u|^{2}dx &  & |u|_{p}^{p}=\int_{\Omega}|u|^{p}dx
\end{eqnarray*}
and we refer to $H_{\varepsilon}$ (resp. $L_{\varepsilon}^{p}$)
as space $H_{0}^{1}(\Omega)$ (resp. $L_{\varepsilon}^{p}$) endowed
with the $\|\cdot\|_{\varepsilon}$ (resp. $|\cdot|_{\varepsilon,p}$)norm.
We refer to the scalar product on $H_{\varepsilon}$ as
\[
\left\langle u,v\right\rangle _{\varepsilon}=\frac{1}{\varepsilon^{3}}\int_{\Omega}\left(\varepsilon^{2}\nabla u\nabla v+uv\right)dx.
\]
Since Schroedinger Maxwell systems are not variational, in a pioneering
paper \cite{BF}, Benci and Fortunato introduced the map $\psi:H_{0}^{1}(\Omega)\rightarrow H_{0}^{1}(\Omega)$
that is the solution of the equation
\begin{equation}
-\Delta\psi(u)=qu^{2}\text{ in }\Omega\label{eq:psi}
\end{equation}
to reduce the system to a single nonlinear variational equation. We
hereafter summarize the main features of the map $\psi(u)$.
\begin{lem}
\label{lem:psi}The map $\psi:H_{0}^{1}(\Omega)\rightarrow H_{0}^{1}(\Omega)$
is positive, of class $C^{2}$ and its derivatives $\psi'(u)$ and
$\psi''(u)$ satisfy 
\begin{eqnarray}
-\Delta\psi'(u)[\varphi] & = & 2qu\varphi\label{eq:derprima}\\
-\Delta\psi''(u)[\varphi_{1},\varphi_{2}] & = & 2q\varphi_{1}\varphi_{2}\label{eq:derseconda}
\end{eqnarray}
for any $\varphi,\varphi_{1},\varphi_{2}\in H_{0}^{1}(\Omega)$.
\end{lem}
\begin{proof}
The proof is standard.
\end{proof}
\begin{rem}
We observe that by simple computation we have that a solution $\psi_{\varepsilon}(u)$
of the equation 
\begin{equation}
\left\{ \begin{array}{ccc}
-\varepsilon^{2}\Delta v=qu^{2} &  & \text{in }\Omega\\
v=0 &  & \text{on \ensuremath{\partial\Omega}}
\end{array}\right.\label{eq:psi-eps}
\end{equation}
can be obtained as $\psi_{\varepsilon}(u)=\psi\left(\frac{u}{\varepsilon}\right).$
The derivatives of $\psi_{\varepsilon}(u)$ thus satisfy
\begin{eqnarray}
-\varepsilon^{2}\Delta\psi_{\varepsilon}'(u)[\varphi] & = & 2qu\varphi\label{eq:derprima-eps}\\
-\varepsilon^{2}\Delta\psi_{\varepsilon}''(u)[\varphi_{1},\varphi_{2}] & = & 2q\varphi_{1}\varphi_{2}\label{eq:derseconda-eps}
\end{eqnarray}
\end{rem}
\begin{lem}
\label{lem:Tder}The map $T_{\varepsilon}:H_{\varepsilon}\rightarrow\mathbb{R}$
given by
\[
T_{\varepsilon}(u)=\int_{\Omega}u^{2}\psi_{\varepsilon}(u)dx
\]
is a $C^{2}$ map and its first derivative is 
\[
T_{\varepsilon}'(u)[\varphi]=4\int_{\Omega}\varphi u\psi_{\varepsilon}(u)dx.
\]
\end{lem}
\begin{proof}
The regularity is standard. The first derivative is 
\[
T_{\varepsilon}'(u)[\varphi]=2\int u\varphi\psi_{\varepsilon}(u)+\int u^{2}\psi_{\varepsilon}'(u)[\varphi].
\]
and by (\ref{eq:derprima-eps}) and (\ref{eq:psi-eps}) we have 
\begin{align*}
2\int u\varphi\psi_{\varepsilon}(u) & =-\frac{1}{q}\left(\varepsilon^{2}\int\Delta(\psi_{\varepsilon}'(u)[\varphi])\psi_{\varepsilon}(u)\right)\\
 & =-\frac{1}{q}\left(\varepsilon^{2}\int\psi_{\varepsilon}'(u)[\varphi]\Delta\psi_{\varepsilon}(u)\right)\\
 & =\int\psi_{\varepsilon}'(u)[\varphi]u^{2},
\end{align*}
so the claim follows.
\end{proof}
Consider the following functional $I_{\varepsilon}\in C^{2}(H_{\varepsilon},\mathbb{R})$.
\begin{equation}
I_{\varepsilon}(u)=\frac{1}{2}\|u\|_{\varepsilon}^{2}+\frac{\omega}{4}G_{\varepsilon}(u)-\frac{1}{p}|u^{+}|_{\varepsilon,p}^{p}\label{eq:ieps}
\end{equation}
where 
\[
G_{\varepsilon}(u)=\frac{1}{\varepsilon^{3}}\int_{\Omega}u^{2}\psi_{\varepsilon}(u)dx=\frac{1}{\varepsilon^{3}}T_{\varepsilon}(u).
\]
By Lemma \ref{lem:Tder} we have 
\[
I_{\varepsilon}'(u)[\varphi]=\frac{1}{\varepsilon^{3}}\int_{\Omega}\varepsilon^{2}\nabla u\nabla\varphi+u\varphi+\omega u\psi_{\varepsilon}(u)\varphi-(u^{+})^{p-1}\varphi
\]
and
\[
I_{\varepsilon}'(u)[u]=\|u\|_{\varepsilon}^{2}+\omega G_{\varepsilon}(u)-|u^{+}|_{\varepsilon,p}^{p};
\]
then if $u$ is a critical points of the functional $I_{\varepsilon}$
the pair of positive functions $(u,\psi_{\varepsilon}(u))$ is a solution
of (\ref{eq:sms}).

We define the following Nehari set
\[
{\mathcal N}_{\varepsilon}=\left\{ u\in H_{0}^{1}(\Omega)\smallsetminus0\ :\ N_{\varepsilon}(u):=I'_{\varepsilon}(u)[u]=0\right\} 
\]
and the infimum level
\begin{equation}
m_{\varepsilon}=\inf_{{\mathcal N}_{\varepsilon}}I_{\varepsilon.}\label{eq:defmeps}
\end{equation}
The Nehari set has the following properties.
\begin{lem}
\label{lem:nehari}If $4<p<6$, ${\mathcal N}_{\varepsilon}$ is a $C^{2}$
manifold and $\inf_{{\mathcal N}_{\varepsilon}}\|u\|_{\varepsilon}>0$. 

Moreover, if $u\in{\mathcal N}_{\varepsilon}$, then 
\begin{align}
I_{\varepsilon}(u) & =\left(\frac{1}{2}-\frac{1}{p}\right)\|u\|_{\varepsilon}^{2}+\omega\left(\frac{1}{4}-\frac{1}{p}\right)G_{\varepsilon}(u)\nonumber \\
 & =\left(\frac{1}{2}-\frac{1}{p}\right)|u^{+}|_{p,\varepsilon}^{p}-\frac{\omega}{4}G_{\varepsilon}(u)\label{eq:I-nehari}\\
 & =\frac{1}{4}\|u\|_{\varepsilon}^{2}+\left(\frac{1}{4}-\frac{1}{p}\right)|u^{+}|_{p,\varepsilon}^{p}.\nonumber 
\end{align}
and it holds Palais-Smale condition for the functional $I_{\varepsilon}$
on ${\mathcal N}_{\varepsilon}$. 

Finally, for all $w\in H_{0}^{1}(\Omega)$ such that $|w^{+}|_{\varepsilon,p}=1$
there exists a unique positive number $t_{\varepsilon}=t_{\varepsilon}(w)$
such that $t_{\varepsilon}(w)w\in{\mathcal N}_{\varepsilon}$. The number
$t_{\varepsilon}$ is the unique critical point of the function 
\[
H(t)=I_{\varepsilon}(tw)=\frac{1}{2}t^{2}\|w\|_{\varepsilon}^{2}+\frac{t^{4}}{4}\omega G_{\varepsilon}(w)-\frac{t^{p}}{p}.
\]
\end{lem}
To obtain the proof of our main theorem, we will perform a blow down
procedure around a point of the domain $\Omega$. 

To perform this procedure, we introduce the \emph{Fermi coordinates}
around a point $\xi\in\partial\Omega$. For $x\in\Omega$ close to
$\xi$ we have $(y(x),t(x))\in\mathbb{R}^{2}\times\mathbb{R}^{+}$
where $t(x)=d(x,\partial\Omega)$ and $y(x)$ are the normal coordinates
of $\partial\Omega$ at $\xi$.

Given $B_{r}^{+}=B(0,r)\times[0,r]$, for a suitable small $r$, the
Fermi coordinates are a diffeomorphism $F_{q}:B_{r}^{+}\rightarrow F_{q}(B_{r}^{+})$.
We call $D(q,r):=F_{q}(B_{r}^{+})$.

In Fermi coordinates we have the following expansion of the scalar
product $g_{ij}(z)$ and of the metric form $|g|^{\frac{1}{2}}$:
\begin{eqnarray}
g^{ij}(z) & = & \delta_{ij}+2h_{ij}t+O(|z|^{2})\text{ for }i,j=1,2\label{eq:g1}\\
g^{i3}(z) & = & \delta_{i3}\label{eq:g2}\\
|g|^{\frac{1}{2}}(z) & = & 1-(n-1)Ht+O(|z|^{2})\label{eq:g3}
\end{eqnarray}
where $h_{ij}$ and $H$ are respectively the second fundamental form
tensor and the mean curvature of $\partial\Omega$ at the point $\xi$.

\subsection{\label{subsec:The-limit-problem}The limit problem }

Consider the following problem in the whole space.
\begin{equation}
\left\{ \begin{array}{cc}
-\Delta u+u+\omega uv=|u|^{p-2}u & \text{ in }\mathbb{R}^{3}\\
-\Delta v=qu^{2} & \text{ in }\mathbb{R}^{3}\\
u>0\text{ in }\mathbb{R}^{3}
\end{array}\right.\label{eq:PL}
\end{equation}
We will refer at problem (\ref{eq:PL}) as the limit problem, in fact
it plays a fundamental role in the blow down procedure hereafter.
We define the function $\psi_{\infty}(u)$ as a solution of the second
equation, and we can reduce the system to a single nonlinear equation.
As pointed out in the introduction, in this special case we know the
explicit expression for $\psi_{\infty}(u)$ which is
\[
\psi_{\infty}(u)=\frac{1}{4\pi}\int_{\mathbb{R}^{3}}\frac{u^{2}(y)}{|x-y|}dy.
\]
As before, we can define a functional
\[
I_{\infty}(u)=\frac{1}{2}\|u\|_{H^{1}}^{2}+\frac{\omega}{4}G(u)-\frac{1}{p}|u^{+}|_{p}^{p}
\]
where $G(u)=\int_{\mathbb{R}^{3}}u^{2}\psi_{\infty}(u)dx$ and the
Nehari manifold 
\[
\mathcal{N}_{\infty}=\left\{ u\in H^{1}(\mathbb{R}^{3})\smallsetminus0\ :\ I_{\infty}'(u)[u]=0\right\} .
\]
 It is possible to prove (see \cite{He}) that the value 
\[
m_{\infty}=\inf_{\mathcal{N}_{\infty}}I_{\infty}
\]
is attained by a positive function $U$ which is a solution of problem
(\ref{eq:PL}), even though the uniquess of the ground state is nowaday
not known. We fix here a positive ground state $U(x)$, and we define
the rescaled function $U_{\varepsilon}(x)=U\left(\frac{x}{\varepsilon}\right)$.

In the following with $U$ we always refer to this particular ground
state we have chosen here. All the proofs work independently of the
choice we made here. 

The function $U_{\varepsilon}$ will be used in section \ref{sec:phieps}
to construct a continuous operator which associate a point in the
domain to a single peaked function in $H_{0}^{1}(\Omega)$. 

While blowing down around an interior point of $\Omega$ leads us
to the limit problem (\ref{eq:PL}), the blow down procedure around
a point of the boundary $\partial\Omega$ gives the following limit
problem on the half space
\begin{equation}
\left\{ \begin{array}{cc}
-\Delta u+u+\omega uv=|u|^{p-2}u & \text{ in }\mathbb{R}_{+}^{3}\\
-\Delta v=qu^{2} & \text{ in }\mathbb{R}_{+}^{3}\\
u,v=0 & \text{ on }\mathbb{\partial R}_{+}^{3}
\end{array}\right.\label{eq:PLhalf}
\end{equation}
We can prove, adapting a result by Esteban and Lions \cite{EL}, that
the only solution of problem (\ref{eq:PLhalf}) is the pair $u\equiv0$
and $v\equiv0$, and this result will be a key argument while proving
concentration results in Section \ref{sec:Conc}. 

The main difference with the theorem of \cite{EL} is that (\ref{eq:PLhalf})
is not a variational system, so the result of Esteban and Lions could
not applied directly; however, reducing (\ref{eq:PLhalf}) to a single
nonlinear equation allows us to prove an analogous result.
\begin{lem}
The system (\ref{eq:PLhalf}) admits only the trivial solution $u\equiv0$,
$v\equiv0$.
\end{lem}
\begin{proof}[Sketch of the proof]

\textbf{Step 1:} We have that $\nabla u\equiv0$ almost everywhere
on $\mathbb{\partial R}_{+}^{3}$. 

Let us define $\psi_{\infty,+}(w)$ the solution of 
\begin{equation}
\left\{ \begin{array}{cc}
-\Delta v=qw^{2} & \text{ in }\mathbb{R}_{+}^{3}\\
v=0 & \text{ on }\partial\mathbb{R}_{+}^{3}
\end{array}\right.\label{eq:psi+}
\end{equation}
and let $u$ the solution of the reduced problem 
\begin{equation}
\left\{ \begin{array}{cc}
-\Delta u+u+\omega u\psi_{\infty,+}(u)=|u|^{p-2}u & \text{ in }\mathbb{R}_{+}^{3}\\
u=0 & \text{ on }\mathbb{\partial R}_{+}^{3}
\end{array}\right..\label{eq:PLhalf-red}
\end{equation}
As in \cite{EL}, we multiply the first equation of (\ref{eq:PLhalf-red})
by ${\displaystyle \frac{\partial u}{\partial x_{n}}}$, obtaining
\begin{align*}
0 & =\int_{\mathbb{R}_{+}^{3}}-\Delta u{\displaystyle \frac{\partial u}{\partial x_{n}}}+u{\displaystyle \frac{\partial u}{\partial x_{n}}}+\omega u{\displaystyle \frac{\partial u}{\partial x_{n}}}\psi_{\infty,+}(u)-|u|^{p-2}u{\displaystyle \frac{\partial u}{\partial x_{n}}dx}\\
 & =\int_{\mathbb{R}_{+}^{3}}\frac{\partial}{\partial x_{n}}\left[\frac{|\nabla u|^{2}+u^{2}}{2}-\frac{|u|^{p}}{p}\right]dx+\omega\int_{\mathbb{R}_{+}^{3}}{\displaystyle \frac{\partial u}{\partial x_{n}}}\psi_{\infty,+}(u)dx.
\end{align*}
Now, in analogy with Lemma \ref{lem:Tder} we have that 
\[
\int_{\mathbb{R}_{+}^{3}}{\displaystyle \frac{\partial u}{\partial x_{n}}}\psi_{\infty,+}(u)dx=\frac{1}{4}\int_{\mathbb{R}_{+}^{3}}{\displaystyle \frac{\partial}{\partial x_{n}}}\left[u^{2}\psi_{\infty,+}(u)\right]dx,
\]
so 
\[
0=\int_{\mathbb{R}_{+}^{3}}\frac{\partial}{\partial x_{n}}\left[\frac{|\nabla u|^{2}+u^{2}}{2}-\frac{|u|^{p}}{p}+\frac{\omega}{4}u^{2}\psi_{\infty,+}(u)\right]dx+\omega\int_{\mathbb{R}_{+}^{3}}{\displaystyle \frac{\partial u}{\partial x_{n}}}\psi_{\infty,+}(u)dx.
\]
 and, by integration by parts and recalling that $u=0$ on $\partial\mathbb{R}_{+}^{3}$,
we get 
\[
0=\int_{\partial\mathbb{R}_{+}^{3}}|\nabla u|^{2}dx_{1}\dots dx_{n-1}
\]
 which proves the claim.

\textbf{Step 2:} We have that $u\in L^{p}(\mathbb{R}_{+}^{3})$ and
$\psi_{\infty,+}(u)\in L^{q}(\mathbb{R}_{+}^{3})$ for every $p\ge2$
and $q\ge6$. Moreover both $u$ and $\psi_{\infty,+}(u)$ are $C^{2}$
functions on $\mathbb{R}_{+}^{3}$.

The proof of this claim is standard.

\textbf{Step 3:} Conclusion

Take a small ball around the origin $B=B(0,r)\subset\mathbb{R}^{3}$
and define on $\mathbb{R}_{+}^{3}\cup B$ the function
\[
\psi_{0}=\left\{ \begin{array}{cc}
\psi_{\infty,+}(u) & \text{ on }\mathbb{R}_{+}^{3}\\
-\psi_{\infty,+}(u) & \text{ on }B\smallsetminus\mathbb{R}_{+}^{3}
\end{array}\right.
\]
so that $\psi_{0}\in C^{2}(\mathbb{R}_{+}^{3}\cup B)$. Now we extend
$u$ to $\mathbb{R}_{+}^{3}\cup B$ as follows:
\[
u_{0}=\left\{ \begin{array}{cc}
u & \text{ on }\mathbb{R}_{+}^{3}\\
0 & \text{ on }B\smallsetminus\mathbb{R}_{+}^{3}
\end{array}\right..
\]
 We have that $u_{0}$ is a $C^{2}$ solution to the equation $-\Delta u+u+\omega u\psi_{0}=|u|^{p-2}u$
in $\mathbb{R}_{+}^{3}\cap B$ which vanishes identically on $B\smallsetminus\mathbb{R}_{+}^{3}$. 

By the unique continuation principle, we can argue that $u_{0}\equiv0$
identically, thus also $u\equiv0$ and, trivially $\psi_{\infty,+}(u)=\psi_{\infty,+}(0)\equiv0$.
\end{proof}

\section{\label{sec:proof}Main ingredient of the proof}

We sketch the proof of Theorem \ref{thm:1}. First of all, it is easy
to see that, if $4<p<6$, the functional $I_{\varepsilon}\in C^{2}$
is bounded below and satisfies the Palais Smale condition on the complete
$C^{2}$ manifold ${\mathcal N}_{\varepsilon}$. We recall a well known
result in nonlinear analysis
\begin{thm}
\label{thm:standard}Let $I\in C^{1,1}(\mathcal{N})$, $\mathcal{N}$
being a $C^{1,1}$ complete Hilbert manifold. If $I$ is bounded from
below on $\mathcal{N}$ and $I$ safisfies the Palais Smale condition,
then $I_{\varepsilon}$ has at least $\cat I^{d}$ critical points
in the sublevel
\[
I^{d}=\left\{ u\in\mathcal{N}\ :\ I(u)\le d\right\} .
\]
 Moreover if $\mathcal{N}$ is contractible and $\cat I^{d}>1$, then
there is at least another critical point $u\notin I^{d}$.
\end{thm}
We prove that, for $\varepsilon$ and $\delta$ small enough, it holds
$\cat\Omega\le\cat\left({\mathcal N}_{\varepsilon}\cap I_{\varepsilon}^{m_{\infty}+\delta}\right)$,
where $m_{\infty}$ has been defined in the previous section. 

To get the inequality $\cat\Omega\le\cat\left({\mathcal N}_{\varepsilon}\cap I_{\varepsilon}^{m_{\infty}+\delta}\right)$
we build two continuous operators
\begin{eqnarray*}
\Phi_{\varepsilon}:\Omega^{-}\rightarrow{\mathcal N}_{\varepsilon}\cap I_{\varepsilon}^{m_{\infty}+\delta} & \text{ and } & \beta:{\mathcal N}_{\varepsilon}\cap I_{\varepsilon}^{m_{\infty}+\delta}\rightarrow\Omega^{+}.
\end{eqnarray*}
where 
\begin{eqnarray*}
\Omega^{-}=\left\{ x\in\Omega\ :\ d(x,\partial\Omega)<r\right\} , &  & \Omega^{+}=\left\{ x\in\mathbb{R}^{3}\ :\ d(x,\partial\Omega)<r\right\} 
\end{eqnarray*}
with $r$ small enough so that $\cat(\Omega^{-})=\cat(\Omega^{+})=\cat(\Omega)$
and such that Definition \ref{def:Fermi} applies.

Following an idea in \cite{BC1}, we build these operators $\Phi_{\varepsilon}$
and $\beta$ such that $\beta\circ\Phi_{\varepsilon}:\Omega^{-}\rightarrow\Omega^{+}$
is homotopic to the immersion $i:\Omega^{-}\rightarrow\Omega^{+}$. 

The operator $\Phi_{\varepsilon}$ is constructed in Section \ref{sec:phieps}
and the definition and the main properties of barycenter map $\beta$
are stated in Section \ref{sec:barycenter}. To define the barycenter
map, however, we have to prove that a low energy function is concetrated
around a point, and that the concentration point can not be to close
to the boundary. These key results are proved in Section \ref{sec:Conc}.

We recall the following well known topological result.
\begin{rem}
\label{rem:cat}Let $X_{1}$ and $X_{2}$, $X_{3}$ be topological
spaces with $X_{1}$ and $X_{3}$ which are homotopically identical.
If $g_{1}:X_{1}\rightarrow X_{2}$ and $g_{2}:X_{2}\rightarrow X_{3}$
are continuous operators such that $g_{2}\circ g_{1}$ is homotopic
to the identity on $X_{1}$, then $\cat X_{1}\leq\cat X_{2}$ . 
\end{rem}
At this point, in light of Remark \ref{rem:cat} we have 
\[
\cat\Omega\le\cat\left({\mathcal N}_{\varepsilon}\cap I_{\varepsilon}^{m_{\infty}+\delta}\right)
\]
and by Theorem \ref{thm:standard} we can conclude that there are
at least $\cat\Omega$ critical points in ${\mathcal N}_{\varepsilon}\cap I_{\varepsilon}^{m_{\infty}+\delta}$.
To conclude the proof of Theorem \ref{thm:1}, in Section \ref{sec:Teps}
we construct a compact contractible set $T_{\varepsilon}$ such that
\[
\Phi_{\varepsilon}(\Omega^{-})\subset T_{\varepsilon}\subset{\mathcal N}_{\varepsilon}\cap I_{\varepsilon}^{C}
\]
 where $C$ is a universal constant (see Lemma \ref{lem:C}). Since
$\Omega$ is non contractible we have $1<\cat\Omega^{-}\le\cat\Phi_{\varepsilon}(\Omega^{-})$
and by Theorem \ref{thm:standard} we conclude the proof. 

\section{\label{sec:Conc}Concentration results}

For any $\varepsilon>0$ we can construct a finite closed partition
$\mathcal{P}^{\varepsilon}=\left\{ P_{j}^{\varepsilon}\right\} _{j\in\Lambda_{\varepsilon}}$
of $\bar{\Omega}$ such that
\begin{itemize}
\item $P_{j}^{\varepsilon}$ is closed for every $j$ and $P_{j}^{\varepsilon}\cap P_{k}^{\varepsilon}\subset\partial P_{j}^{\varepsilon}\cap\partial P_{k}^{\varepsilon}$
for $j\neq k$;
\item there exist $r_{1}(\varepsilon),r_{2}(\varepsilon)>0$, with $C_{1}\varepsilon\le r_{1}(\varepsilon)<r_{2}(\varepsilon)\le C_{2}\varepsilon$
for some positive constants $C_{1},C_{2}$, and a positive number
$K$, such that, if $P_{j}^{\varepsilon}\cap\partial M=\emptyset$,
then there are points $q_{j}^{\varepsilon}\in P_{j}^{\varepsilon}$
for which  $B(q_{j}^{\varepsilon},r_{1}(\varepsilon))\subset P_{j}^{\varepsilon}\subset B(q_{j}^{\varepsilon},r_{2}(\varepsilon))\subset B(q_{j}^{\varepsilon},Kr_{2}(\varepsilon))$,
while, if $P_{j}^{\varepsilon}\cap\partial M\neq\emptyset$, then
there are points $q_{j}^{\varepsilon}\in P_{j}^{\varepsilon}\cap\partial\Omega$
for which $D(q_{j}^{\varepsilon},r_{1}(\varepsilon))\subset P_{j}^{\varepsilon}\subset D(q_{j}^{\varepsilon},r_{2}(\varepsilon))\subset D(q_{j}^{\varepsilon},Kr_{2}(\varepsilon))$.
To simplify the notations we set $A_{j}^{\varepsilon}(r)=\left\{ \begin{array}{ccc}
B(q_{j}^{\varepsilon},r) &  & P_{j}^{\varepsilon}\cap\partial M=\emptyset\\
D(q_{j}^{\varepsilon},r) &  & P_{j}^{\varepsilon}\cap\partial M\neq\emptyset
\end{array}\right.$
\item lastly, there exists a finite number $\nu(\Omega)\in\mathbb{N}$ such
that every $q\in\bar{\Omega}$ is contained in at most $\nu(\Omega)$
sets $A_{j}^{\varepsilon}(Kr_{2}(\varepsilon))$, where $\nu(\Omega)$
does not depends on $\varepsilon$.
\end{itemize}
By compactness of $\bar{\Omega}$ such a partition exists, at least
for small $\varepsilon$. In the following we will choose always $\varepsilon_{0}(\delta)$
sufficiently small in order to have this partition. We remark that
such a partiton can be obtained in this case simply by splitting the
whole space in cubes $\left\{ Q_{j}^{\varepsilon}\right\} _{j\in\Lambda_{\varepsilon}}$
with sides of lenght $\varepsilon$ and taking $\mathcal{P}^{\varepsilon}=\left\{ Q_{j}^{\varepsilon}\cap\bar{\Omega}\right\} _{j\in\Lambda_{\varepsilon}}$.
We prefer to state the general properties of the partiton $\mathcal{P}^{\varepsilon}$
since this could be a non trivial generalization when dealing with
Riemannian manifolds with boundary. 
\begin{lem}
\label{lem:gamma} We recall that there exists a constant $\gamma>0$
such that, for any $\delta>0$ and for any $\varepsilon<\varepsilon_{0}(\delta)$
as in Proposition \ref{prop:phieps}, given any ``good'' partition
${\mathcal P}_{\varepsilon}=\left\{ P_{j}^{\varepsilon}\right\} _{j}$
of the domain $\bar{\Omega}$ and for any function $u\in{\mathcal N}_{\varepsilon}\cap I_{\varepsilon}^{m_{\varepsilon}+\delta}$
there exists, for an index $\bar{j}$ a set $P_{\bar{j}}^{\varepsilon}$
such that 
\begin{equation}
\frac{1}{\varepsilon^{3}}\int_{P_{\bar{j}}^{\varepsilon}}|u^{+}|^{p}dx\ge\gamma.\label{eq:gamma}
\end{equation}
\end{lem}
\begin{proof}
We follows the proof of Lemma 5.3 of \cite{BBM} and of Lemma 10 of
\cite{GM10}. By Remark \ref{rem:limsup} we have that $\mathcal{N}_{\varepsilon}\cap I_{\varepsilon}^{m_{\varepsilon}+\delta}\neq\emptyset$.
For any function $u\in\mathcal{N}_{\varepsilon}\cap I_{\varepsilon}^{m_{\varepsilon}+\delta}$
we denote by $u_{j}^{+}$ the restriction of $u^{+}$ to the set $P_{j}^{\varepsilon}$.
Then we can write, since $I'(u)[u]=0$, 
\begin{align*}
\|u\|_{\varepsilon}^{2} & =|u^{+}|_{\varepsilon,p}^{p}-\frac{1}{\varepsilon^{3}}\int_{\Omega}\omega u^{2}\psi(u)\le|u^{+}|_{\varepsilon,p}^{p}\\
 & =\sum_{j}|u_{j}^{+}|_{\varepsilon,p}^{p}\le\max_{j}\left\{ |u_{j}^{+}|_{\varepsilon,p}^{p-2}\right\} \sum_{j}|u_{j}^{+}|_{\varepsilon,p}^{2}.
\end{align*}
We define the functions $\tilde{u}_{j}$ by using a smooth real cutoff
function $\chi_{\varepsilon}^{j}:\mathbb{R}^{3}\rightarrow[0,1]$
such that $|\nabla\chi_{\varepsilon}^{j}|\leq\frac{2}{\bar{K}\varepsilon}$
for some constant $\bar{K}$ and, $\chi_{\varepsilon}^{j}=1$ for
$x\in A_{j}^{\varepsilon}(r_{2}(\varepsilon))$ and $\chi_{\varepsilon}^{j}=0$
for $x$ outside $A_{j}^{\varepsilon}(Kr_{2}(\varepsilon))$. Also,
we define $\tilde{u}_{j}(x)=u^{+}(x)\chi_{\varepsilon}^{j}(x)$. It
holds $\tilde{u}_{j}\in H_{0}^{1}(\Omega)$, hence there exists a
positive constant $C$ such that, for any $j$, $|u_{j}^{+}|_{\varepsilon,p}^{2}\leq|\tilde{u}_{j}|_{\varepsilon,p}^{2}\leq C||\tilde{u}_{j}||_{\varepsilon}^{2}=C||\tilde{u}_{j}||_{\varepsilon,P_{j}^{\varepsilon}}^{2}+C||\tilde{u}_{j}||_{\varepsilon,A_{j}^{\varepsilon}(r_{2}(\varepsilon))\smallsetminus P_{j}^{\varepsilon}}^{2}$.
Moreover 
\begin{eqnarray*}
\int_{A_{j}^{\varepsilon}(r_{2}(\varepsilon))\smallsetminus P_{j}^{\varepsilon}}\left\vert \tilde{u}_{j}\right\vert ^{2}dx & \leq & \int_{A_{j}^{\varepsilon}(r_{2}(\varepsilon))\smallsetminus P_{j}^{\varepsilon}}\left\vert u^{+}\right\vert ^{2}dx;\\
\int_{A_{j}^{\varepsilon}(r_{2}(\varepsilon))\smallsetminus P_{j}^{\varepsilon}}\varepsilon^{2}\left\vert \nabla\tilde{u}_{j}\right\vert ^{2}dx & \leq & \int_{A_{j}^{\varepsilon}(r_{2}(\varepsilon))\smallsetminus P_{j}^{\varepsilon}}\left(\varepsilon^{2}\left\vert \nabla u^{+}\right\vert ^{2}+\frac{4}{\bar{K}^{2}}\left\vert u^{+}\right\vert ^{2}\right)dx.
\end{eqnarray*}
 Hence we obtain 
\begin{eqnarray*}
\sum_{j}|u_{j}^{+}|_{\varepsilon,p}^{2} & \leq & C\sum_{j}\left\vert \left\vert u^{+}\right\vert \right\vert _{\varepsilon}^{2}+C\left(\frac{4}{\bar{K}^{2}}+1\right)\nu(\Omega)\left\vert \left\vert u^{+}\right\vert \right\vert _{\varepsilon}^{2}\leq\\
 & \leq & C\left(\frac{4}{\bar{K}^{2}}+2\right)\nu(\Omega)\left\vert \left\vert u\right\vert \right\vert _{\varepsilon}^{2}.
\end{eqnarray*}
 We can conclude that 
\[
\max_{j}\left\{ \left(\frac{1}{\varepsilon^{3}}\int_{P_{j}^{\varepsilon}}\left\vert u^{+}\right\vert ^{p}dx\right)^{\frac{p-2}{p}}\right\} \geq\frac{1}{C\left(\frac{4}{\bar{K}^{2}}+2\right)\nu(\Omega)},
\]
 so the proof is complete. 
\end{proof}
\begin{lem}
\label{lem:crucial}Let $\gamma>0$. Suppose that there exist a sequence
of functions $u_{k}\in\mathcal{N}_{\varepsilon_{k}}\cap I_{\varepsilon_{k}}^{m_{\varepsilon_{k}}+\delta_{k}}$
with $\varepsilon_{k},\delta_{k}\rightarrow0$ and \textup{$\left|I'_{\varepsilon_{k}}(u_{k})[\varphi]\right|\le\sigma_{k}\|\varphi\|_{\varepsilon_{k}}$
where $\sigma_{k}\rightarrow0$ and a sequence of sets }$P_{k}^{\varepsilon_{k}}\subset{\mathcal P}_{\varepsilon_{k}}$
such that $\frac{1}{\varepsilon_{k}^{3}}\int_{P_{k}^{\varepsilon_{k}}}|u_{k}^{+}|^{p}dx\ge\gamma$.
Then 
\[
\liminf_{k\rightarrow+\infty}\frac{d(\partial\Omega,P_{k}^{\varepsilon_{k}})}{\varepsilon_{k}}=+\infty.
\]
\end{lem}
\begin{proof}
By contradiction, suppose that, up to subsequence, $\frac{d(\partial\Omega,P_{k}^{\varepsilon_{k}})}{\varepsilon_{k}}\le R$
for some $R>0$. Take $q_{k}\in\partial\Omega$ such that $P_{k}^{\varepsilon_{k}}\subset D(q_{k},C\varepsilon_{k})\subset D(q_{k},r)$,
for some $C>0$ which does not depend on $\varepsilon_{k}$. This
is possible since $\frac{d(\partial\Omega,P_{k}^{\varepsilon_{k}})}{\varepsilon_{k}}\le R$
and the diameter of each $P_{k}^{\varepsilon_{k}}$ is bounded by
$C_{2}\varepsilon_{k}$ with $C_{2}$ independent of $\varepsilon_{k}$
. On $D(q_{k},r)$ we define the Fermi coordinates (see Def. \ref{def:Fermi})
\[
F_{q_{z}}(z):B_{r}^{+}\rightarrow D(q_{k},r).
\]
 Define a smooth cut off function $\chi:\mathbb{R}_{+}^{3}\rightarrow\mathbb{R}$
such that 
\[
\chi(z):=\left\{ \begin{array}{ccc}
1 &  & z\in B_{r/2}^{+}\\
0 &  & z\in\mathbb{R}_{+}^{3}\smallsetminus B_{r}^{+}
\end{array}\right.
\]
and the function $w_{k}:\mathbb{R}_{+}^{3}\rightarrow\mathbb{R}$
\[
w_{k}(z)=\left\{ \begin{array}{ccc}
u_{k}\left(F_{q_{k}}(\varepsilon_{k}z)\right)\chi(\varepsilon_{k}z) &  & z\in B_{r/\varepsilon_{k}}^{+}\\
0 &  & \text{elsewhere}
\end{array}\right..
\]
We have that $w_{k}\in H_{0}^{1}(\mathbb{R}_{+}^{3})$ and, by simple
computation, that $\|w_{k}\|_{H^{1}(\mathbb{R}_{+}^{3})}\le C\|u_{k}\|_{\varepsilon_{k}}\le C$
so $w_{k}$ converges to some function $w\in H_{0}^{1}(\mathbb{R}_{+}^{3})$,
weakly in $H_{0}^{1}(\mathbb{R}_{+}^{3})$ and strongly in $L_{\text{loc}}^{s}(\mathbb{R}_{+}^{3})$
for $2\le s<6.$ 

Let $\psi_{k}(x):=\psi_{\varepsilon_{k}}(u_{k})(x)$ where $\psi_{\varepsilon_{k}}(u_{k})$
solves $-\varepsilon_{k}^{2}\Delta v=qu_{k}^{2}\text{ in }\Omega$,
as defined in eq. (\ref{eq:psi-eps}) and define $\tilde{\psi}_{k}:\mathbb{R}_{+}^{3}\rightarrow\mathbb{R}$
as 
\[
\tilde{\psi}_{k}(z)=\left\{ \begin{array}{ccc}
\psi_{k}\left(F_{q_{k}}(\varepsilon_{k}z)\right)\chi(\varepsilon_{k}z) &  & z\in B_{r/\varepsilon_{k}}^{+}\\
0 &  & \text{elsewhere}
\end{array}\right..
\]
 and again $|\nabla\tilde{\psi}_{k}|_{L^{2}(\mathbb{R}_{+}^{3})}\le C\varepsilon_{k}^{2}|\nabla\psi_{k}|_{\varepsilon_{k},2}^{2}$.
Moreover, using that $\psi_{k}$ solves (\ref{eq:psi-eps}) we have
\[
\varepsilon_{k}^{2}|\nabla\psi_{k}|_{\varepsilon_{k},2}^{2}=\frac{\varepsilon_{k}^{2}}{\varepsilon_{k}^{3}}\int_{\Omega}|\nabla\psi_{k}|^{2}=\frac{1}{\varepsilon_{k}^{3}}q\int_{\Omega}u_{k}^{2}\psi_{k}\le C|u_{k}|_{\varepsilon,12/5}^{2}|\psi_{k}|_{\varepsilon,6}^{2}
\]
 so $\|\tilde{\psi}_{k}\|_{D^{1,2}(\mathbb{R}_{+}^{3})}\le C$ and
$\tilde{\psi}_{k}\rightharpoonup\bar{\psi}$ weakly in $D^{1,2}(\mathbb{R}_{+}^{3})$
and in $L^{6}(\mathbb{R}_{+}^{3})$ for some function $\bar{\psi}\in D^{1,2}(\mathbb{R}_{+}^{3})$
with $\bar{\psi}=0$ on $\partial\mathbb{R}_{+}^{3}$. We show that
$\bar{\psi}=\psi_{\infty,+}(w)$, where $\psi_{\infty,+}(w)$ is defined
in (\ref{eq:psi+})

Take a function $f\in C_{0}^{\infty}(\mathbb{R}_{+}^{3})$. We have
$\text{supp}(f)\subset B^{+}(0,T)\subset B_{r/2\varepsilon_{k}}^{+}$,
for some $T>0$, and for $k$ sufficently large. Here we denote by
$\text{supp}(f)$ the support of the function $f$. We define 
\[
f_{k}(x)=f\left(\frac{1}{\varepsilon_{k}}F_{q_{k}}^{-1}(x)\right)\in H_{0}^{1}(\Omega),
\]
so $\text{supp}(f_{k})\subset D(q_{k},r/2)$. Since $\psi_{k}=\psi_{\varepsilon_{k}}(u_{k})$
is a weak solution of (\ref{eq:psi-eps}), we have $\varepsilon_{k}^{2}\int_{\Omega}\nabla\psi_{k}\nabla f_{k}dx=q\int_{\Omega}u_{k}^{2}f_{k}dx$.
Now, by means of Fermi coordinates, with the change of variables $x=F_{q_{k}}(\varepsilon_{k}z)$
we have 
\begin{multline*}
\frac{1}{\varepsilon_{k}}\int_{\Omega}\nabla\psi_{k}\nabla f_{k}dx=\frac{1}{\varepsilon_{k}}\int_{D(q_{k},r/2)}\nabla\psi_{k}\nabla f_{k}dx=\int_{B_{r/2\varepsilon_{k}}^{+}}g_{ij}(\varepsilon_{k}z)\partial_{i}\tilde{\psi}_{k}\partial_{j}f|g(\varepsilon_{k}z)|^{\frac{1}{2}}dz\\
=\int_{\text{supp}(f)}\nabla\tilde{\psi}_{k}\nabla fdz+O(\varepsilon_{k})=\int_{\mathbb{R}_{+}^{3}}\nabla\tilde{\psi}_{k}\nabla fdz+O(\varepsilon_{k})\rightarrow\int_{\mathbb{R}_{+}^{3}}\nabla\bar{\psi}\nabla fdz.
\end{multline*}
In the same way 
\begin{multline*}
\frac{1}{\varepsilon_{k}^{3}}\int_{\Omega}u_{k}^{2}f_{k}dx=\frac{1}{\varepsilon_{k}^{3}}\int_{D(q_{k},r/2)}u_{k}^{2}f_{k}dxdx=\int_{B_{r/2\varepsilon_{k}}^{+}}w_{k}^{2}f|g(\varepsilon_{k}z)|^{\frac{1}{2}}dz\\
=\int_{\text{supp}(f)}w_{k}^{2}fdz+O(\varepsilon_{k})=\int_{\mathbb{R}_{+}^{3}}w_{k}^{2}fdz+O(\varepsilon_{k})\rightarrow\int_{\mathbb{R}_{+}^{3}}w^{2}fdz,
\end{multline*}
hence for any $f\in C_{0}^{\infty}(\mathbb{R}_{+}^{3})$ 
\[
\int_{\mathbb{R}_{+}^{3}}\nabla\bar{\psi}\nabla fdz=q\int_{\mathbb{R}_{+}^{3}}w^{2}fdz
\]
and we have proved that $\bar{\psi}=\psi_{\infty,+}(w)$, as claimed. 

In a similar way we want to prove that $w$ solves weakly 
\[
\left\{ \begin{array}{cc}
-\Delta w+w+\omega w\psi_{\infty,+}(w)=|w|^{p-2}w & \text{ in }\mathbb{R}_{+}^{3}\\
u,v=0 & \text{ on }\partial\mathbb{R}_{+}^{3}
\end{array}\right..
\]
Again, we take a function $f\in C_{0}^{\infty}(\mathbb{R}_{+}^{3})$
and in the same way we define $f_{k}(x)=f\left(\frac{1}{\varepsilon_{k}}F_{q_{k}}^{-1}(x)\right)$.
By hypothesis on $u_{k}$, we have $\left|I'_{\varepsilon_{k}}(u_{k})[f_{k}]\right|\le\sigma_{k}\|f_{k}\|_{\varepsilon_{k}}\le C\|f\|_{H^{1}(\mathbb{R}_{+}^{3})}\text{ where }\sigma_{k}\rightarrow0$,
we obtain
\begin{align*}
I_{\varepsilon_{k}}'(u_{k})[f_{k}] & =\frac{1}{\varepsilon_{k}^{3}}\int_{\Omega}\varepsilon_{k}^{2}\nabla u_{k}\nabla f_{k}+u_{k}f_{k}+\omega u_{k}\psi_{\varepsilon_{k}}(u_{k})f_{k}-(u_{k}^{+})^{p-1}f_{k}dx\\
 & =\frac{1}{\varepsilon_{k}^{3}}\int_{D(q_{k},r/2)}\varepsilon_{k}^{2}\nabla u_{k}\nabla f_{k}+u_{k}f_{k}+\omega u_{k}\psi_{\varepsilon_{k}}(u_{k})f_{k}-(u_{k}^{+})^{p-1}f_{k}dx\\
 & =\int_{\text{supp}(f)}\left(g_{ij}(\varepsilon_{k}z)\partial_{i}w_{k}\partial_{j}f+w_{k}f+\omega w_{k}\tilde{\psi}_{k}f-(w_{k}^{+})^{p-1}f\right)|g(\varepsilon_{k}z)|^{\frac{1}{2}}dx\\
 & =\int_{\text{supp}(f)}\nabla w_{k}\nabla f+w_{k}f+\omega w_{k}\tilde{\psi}_{k}f-(w_{k}^{+})^{p-1}fdx+O(\varepsilon_{k})
\end{align*}
and, since $w_{k}\rightarrow w$ and $\tilde{\psi}_{k}\rightarrow\psi_{\infty,+}(w)$
strongly in $L^{s}(\text{supp}(f))$ for $2\le s<6$ and $I_{\varepsilon_{k}}'(u_{k})[f_{k}]\rightarrow0,$
we conclude, as claimed, that for any $f\in C_{0}^{\infty}(\mathbb{R}_{+}^{3})$
\begin{align*}
0 & =\int_{\text{supp}(f)}\nabla w\nabla f+wf+\omega w\psi_{\infty,+}(w)f-(w^{+})^{p-1}fdx\\
 & =\int_{\mathbb{R}_{+}^{3}}\nabla w\nabla f+wf+\omega w\psi_{\infty,+}(w)f-(w^{+})^{p-1}fdx,
\end{align*}
so the pair $(w,\psi_{\infty,+}(w))$ is a solution of (\ref{eq:PLhalf}).
By {[}Esteban-Lions{]}, we have then that $(w,\psi_{\infty}(w))\equiv(0,0)$.
Thus $w_{k}\rightarrow0$ strongly in $L_{\text{loc}}^{s}(\mathbb{R}_{+}^{3})$
for $2\le s<6$. This gives us the contradiction, indeed, 
\[
0<\gamma\le\frac{1}{\varepsilon_{k}^{3}}\int_{P_{k}^{\varepsilon_{k}}}|u_{k}^{+}|^{p}dx\le\frac{1}{\varepsilon_{k}^{3}}\int_{D(q_{k},K\varepsilon_{k})}|u_{k}|^{p}dx=\int_{B_{K}^{+}}|w_{k}|^{p}dz+O(\varepsilon_{k})\rightarrow0
\]
 since $w_{k}\rightarrow0$ strongly in $L^{p}(B_{K}^{+})$. This
ends the proof. 
\end{proof}

\section{\label{sec:phieps}The map $\Phi_{\varepsilon}$ }

For every $\xi\in\Omega^{-}$ we define the function
\begin{equation}
W_{\xi,\varepsilon}(x)=U_{\varepsilon}(x-\xi)\chi(|x-\xi|)\label{eq:Weps}
\end{equation}
where $\chi:\mathbb{R}^{+}\rightarrow\mathbb{R}^{+}$ is a cut off
function, that is $\chi\equiv1$ for $t\in[0,r/2)$, $\chi\equiv0$
for $t>r$ and $|\chi'(t)|\le2/r$.

We can define a map
\begin{eqnarray*}
\Phi_{\varepsilon}:\Omega^{-}\rightarrow{\mathcal N}_{\varepsilon} & ; & \Phi_{\varepsilon}(\xi)=t_{\varepsilon}(W_{\xi,\varepsilon})W_{\xi,\varepsilon}
\end{eqnarray*}

\begin{rem}
\label{w}The following limits hold uniformly with respect to $\xi\in\Omega$
\begin{eqnarray*}
\|W_{\varepsilon,\xi}\|_{\varepsilon} & \rightarrow & \|U\|_{H^{1}(\mathbb{R}^{3})}\\
|W_{\varepsilon,\xi}|_{\varepsilon,t} & \rightarrow & \|U\|_{L^{t}(\mathbb{R}^{3})}\text{ for all }2\le t\le6
\end{eqnarray*}
\end{rem}
\begin{lem}
\label{lem:stimaGeps}The following limit holds uniformly with respect
to $\xi\in\Omega$:
\[
\lim_{\varepsilon\rightarrow0}G_{\varepsilon}(W_{\varepsilon,\xi})=G(U)=\int_{\mathbb{R}^{3}}qU^{2}\psi(U)dx
\]
\end{lem}
\begin{proof}
To simplify the notation, set $\psi_{\varepsilon}(x):=\psi_{\varepsilon}(W_{\xi,\varepsilon})(x)$.
By definition, $\psi_{\varepsilon}(x)$ solves
\[
-\varepsilon^{2}\Delta\psi_{\varepsilon}=qW_{\xi,\varepsilon}^{2}.
\]
Also, let us define $\tilde{\psi}_{\varepsilon}(z)=\psi_{\varepsilon}(\varepsilon z+\xi)$

By change of variables, we have that 
\begin{equation}
-\Delta_{z}\tilde{\psi}_{\varepsilon}(z)=-\Delta_{z}\psi_{\varepsilon}(\varepsilon z+\xi)=-\varepsilon^{2}\left(\Delta\psi_{\varepsilon}\right)(\varepsilon z+\xi)=qW_{\xi,\varepsilon}^{2}(\varepsilon z+\xi)=qU^{2}(z)\chi^{2}(\varepsilon z)\label{eq:psitilde}
\end{equation}
for any $z$ such that $\varepsilon z+\xi\in\Omega$. Let us call
\[
\Omega_{\varepsilon}=\left\{ z\in\mathbb{R}^{3}\ :\ \varepsilon z+\xi\in\Omega\right\} .
\]
Since $\xi\in\Omega^{-}$ we have that $B(0,r/\varepsilon)\subset\Omega_{\varepsilon}$
so, as $\varepsilon\rightarrow0$, $\Omega_{\varepsilon}\nearrow\mathbb{R}^{3}$
. Also we extend $\tilde{\psi}_{\varepsilon}$ trivially by 0 outside
$\Omega_{\varepsilon}$ (with abuse of notation we still call the
extension $\tilde{\psi}_{\varepsilon}$)

By (\ref{eq:psitilde}), we have that 
\[
\|\nabla\tilde{\psi}_{\varepsilon}\|_{2}^{2}=q\int U^{2}(z)\chi^{2}(\varepsilon z)\tilde{\psi}_{\varepsilon}(z)\le C\|U\|_{12/5}^{2}\|\tilde{\psi}_{\varepsilon}\|_{6}
\]
that implies that $\tilde{\psi}_{\varepsilon}$ is bounded in $D^{1,2}(\mathbb{R}^{3})$.
So there exists a $\bar{\psi}\in D^{1,2}(\mathbb{R}^{3})$ such that
\[
\tilde{\psi}_{\varepsilon}\rightharpoonup\bar{\psi}\text{ in }D^{1,2}(\mathbb{R}^{3})\text{ and in }L^{6}(\mathbb{R}^{3})\text{ while }\varepsilon\rightarrow0.
\]
We have that $\bar{\psi}$ is a weak solution of $-\Delta v=qU^{2}$,
that is $\bar{\psi}=\psi_{\infty}(U)$. In fact, for any $\varphi\in C_{0}^{\infty}(\mathbb{R}^{3})$,
we have that the support of $\varphi$ is eventually contained in
$\Omega_{\varepsilon}$ and it holds
\begin{align*}
\int_{\mathbb{R}^{3}}\nabla\tilde{\psi}_{\varepsilon}\nabla\varphi & =\int_{\Omega_{\varepsilon}}\nabla\tilde{\psi}_{\varepsilon}\nabla\varphi=-\int_{\Omega_{\varepsilon}}\Delta\tilde{\psi}_{\varepsilon}\varphi=q\int_{\Omega_{\varepsilon}}U^{2}(z)\chi^{2}(\varepsilon z)\varphi\\
 & =q\int_{\mathbb{R}^{3}}U^{2}(z)\chi^{2}(\varepsilon z)\varphi\rightarrow q\int_{\mathbb{R}^{3}}U^{2}(z)\varphi
\end{align*}
as $\varepsilon\rightarrow0$, thus $\int_{\mathbb{R}^{3}}\nabla\bar{\psi}\nabla\varphi=q\int_{\mathbb{R}^{3}}U^{2}(z)\varphi$
and $\bar{\psi}=\psi_{\infty}(U)$, as claimed.

Now we can conclude that
\begin{align*}
G_{\varepsilon}(W_{\varepsilon,\xi}) & =\frac{1}{\varepsilon^{3}}\int_{\Omega}W_{\xi,\varepsilon}^{2}(x)\psi_{\varepsilon}(x)dx=\int_{\mathbb{R}^{3}}W_{\xi,\varepsilon}^{2}(\varepsilon z+\xi)\psi_{\varepsilon}(\varepsilon z+\xi)dy\\
 & =\int_{\mathbb{R}^{3}}U^{2}(z)\chi^{2}(\varepsilon z)\tilde{\psi}_{\varepsilon}(z)dz\rightarrow\int_{\mathbb{R}^{3}}U^{2}\psi_{\infty}(U)dz
\end{align*}
since $\tilde{\psi}_{\varepsilon}\rightharpoonup\psi_{\infty}(U)$
in $L^{6}(\mathbb{R}^{3})$ and $U(x)\chi(\varepsilon x)\rightarrow U(x)$
in $L^{12/5}(\mathbb{R}^{3})$. 
\end{proof}
\begin{prop}
\label{prop:phieps}For all $\varepsilon>0$ the map $\Phi_{\varepsilon}$
is continuous. Moreover for any $\delta>0$ there exists $\varepsilon_{0}=\varepsilon_{0}(\delta)$
such that, if $\varepsilon<\varepsilon_{0}$ then $I_{\varepsilon}\left(\Phi_{\varepsilon}(\xi)\right)<m_{\infty}+\delta$.
\end{prop}
\begin{proof}
It is easy to see that $\Phi_{\varepsilon}$ is continuous because
$t_{\varepsilon}(w)$ depends continuously on $w\in H_{0}^{1}$.

At this point we prove that $t_{\varepsilon}(W_{\varepsilon,\xi})\rightarrow1$
uniformly with respect to $\xi\in\Omega$. In fact, by Lemma \ref{lem:nehari},
$t_{\varepsilon}(W_{\varepsilon,\xi})$ is the unique solution of
\[
\|W_{\varepsilon,\xi}\|_{\varepsilon}^{2}+t^{2}\omega G_{\varepsilon}(W_{\varepsilon,\xi})-t^{p-2}|W_{\varepsilon,\xi}|_{\varepsilon,p}^{p}=0.
\]
By Remark \ref{w} and Lemma \ref{lem:stimaGeps} we have the claim.
In fact, we recall that, since $U$ is a solution of (\ref{eq:PL})
it holds $\|U\|_{H^{1}(\mathbb{R}^{3})}^{2}+\omega G(U)-|U|_{p}^{p}=0$.

At this point, we have 
\[
I_{\varepsilon}\left(t_{\varepsilon}(W_{\varepsilon,\xi})W_{\varepsilon,\xi}\right)=\left(\frac{1}{2}-\frac{1}{p}\right)\|W_{\varepsilon,\xi}\|_{\varepsilon}^{2}t_{\varepsilon}^{2}+\omega\left(\frac{1}{4}-\frac{1}{p}\right)t_{\varepsilon}^{4}G_{\varepsilon}(W_{\varepsilon,\xi})
\]
Again, by Remark \ref{w} and Lemma \ref{lem:stimaGeps} and since
$t_{\varepsilon}(W_{\varepsilon,\xi})\rightarrow1$ we have

\[
I_{\varepsilon}\left(t_{\varepsilon}(W_{\varepsilon,\xi})W_{\varepsilon,\xi}\right)\rightarrow\left(\frac{1}{2}-\frac{1}{p}\right)\|U\|_{H^{1}(\mathbb{R}^{3})}^{2}+\omega\left(\frac{1}{4}-\frac{1}{p}\right)G(U)=m_{\infty}
\]
that concludes the proof.
\end{proof}
\begin{rem}
\label{rem:limsup}By Proposition \ref{prop:phieps} we have that 

\begin{equation}
\limsup_{\varepsilon\rightarrow0}m_{\varepsilon}\le m_{\infty.}\label{eq:limsup}
\end{equation}
\end{rem}

\section{\label{sec:barycenter}The map $\beta$}

For any $u\in{\mathcal N}_{\varepsilon}$ we can define a point $\beta(u)\in\mathbb{R}^{3}$
by 
\[
\beta(u)=\frac{\int_{\Omega}x\Gamma(u)dx}{\int_{\Omega}\Gamma(u)dx}
\]
where $\Gamma(u)=\frac{1}{4}\left[\frac{1}{\varepsilon}|\nabla u|^{2}+\frac{1}{\varepsilon^{3}}|u|^{2}\right]+\left(\frac{1}{4}-\frac{1}{p}\right)\frac{1}{\varepsilon^{3}}|u^{+}|^{p}$.
We notice that, since $4<p<6,$ $\Gamma(u)\ge0$. 
\begin{lem}
The function $\beta$ is well defined in ${\mathcal N}_{\varepsilon}$. 
\end{lem}
\begin{proof}
We have that $\int_{\Omega}\Gamma(u)dx=I_{\varepsilon}(u)\ge m_{\varepsilon}$
if $u\in{\mathcal N}_{\varepsilon}$. So, we want to prove that $m_{\varepsilon}\ge\alpha$
for some $\alpha>0$. 

Take $w$ such that $|w|_{\varepsilon,p}=1$, and $t_{\varepsilon}=t_{\varepsilon}(w)$
such that $t_{\varepsilon}w\in{\mathcal N}_{\varepsilon}$. By (\ref{eq:I-nehari})
we have
\[
I_{\varepsilon}(t_{\varepsilon}w)=\frac{t_{\varepsilon}^{2}}{4}\|w\|_{\varepsilon}^{2}+\left(\frac{1}{4}-\frac{1}{p}\right)t_{\varepsilon}^{p}\ge\left(\frac{1}{4}-\frac{1}{p}\right)t_{\varepsilon}^{p}.
\]
Moreover, we have that $\inf_{|w|_{\varepsilon,p}=1}t_{\varepsilon}(w)>0$.
In fact, suppose that there exists a sequence $w_{n}$ such that $|w_{n}|_{\varepsilon,p}=1$
and $t_{\varepsilon}(w_{n})\rightarrow0$. Since $t_{\varepsilon}(w_{n})w_{n}\in{\mathcal N}_{\varepsilon}$
it holds
\[
1=|w_{n}|_{\varepsilon,p}=\frac{1}{t_{\varepsilon}(w_{n})^{p-2}}\|w_{n}\|_{\varepsilon}^{2}+\omega G_{\varepsilon}(t_{\varepsilon}(w_{n}))\ge\frac{1}{t_{\varepsilon}(w_{n})^{p-2}}\|w_{n}\|_{\varepsilon}^{2}.
\]
Also, we have that there exists a constant $C>0$ which does not depend
on $\varepsilon$ such that $|w_{n}|_{\varepsilon,p}\le C\|w_{n}\|_{\varepsilon}$,
so 
\[
1\ge\frac{1}{Ct_{\varepsilon}(w_{n})^{p-2}}\rightarrow+\infty
\]
that is a contradiction. This proves that $m_{\varepsilon}\ge\alpha$
for some $\alpha>0$ and hence that $\beta$ is well defined in ${\mathcal N}_{\varepsilon}$.
\end{proof}
Now we have to prove that, if $u\in{\mathcal N}_{\varepsilon}\cap I_{\varepsilon}^{m_{\infty}+\delta}$
then $\beta(u)\in\Omega^{+}$.
\begin{prop}
\label{prop:conc}For any $\eta\in(0,1)$ there exists $\delta_{0}<m_{\infty}$
such that for any $\delta\in(0,\delta_{0})$ and any $\varepsilon\in(0,\varepsilon_{0}(\delta))$
as in Proposition \ref{prop:phieps}, for any function $u\in{\mathcal N}_{\varepsilon}\cap I_{\varepsilon}^{m_{\infty}+\delta}$
we can find a point $q=q(u)\in\Omega$ such that 
\[
\int_{B(q,r/2)\cap\Omega}\Gamma(u)>\left(1-\eta\right)m_{\infty}.
\]
\end{prop}
\begin{proof}
First, we prove the proposition for $u\in{\mathcal N}_{\varepsilon}\cap I_{\varepsilon}^{m_{\varepsilon}+2\delta}$. 

By contradiction, we assume that there exists $\eta\in(0,1)$ such
that we can find two sequences of vanishing real number $\delta_{k}$
and $\varepsilon_{k}$ and a sequence of functions $\left\{ u_{k}\right\} _{k}$
such that $u_{k}\in{\mathcal N}_{\varepsilon_{k}}$, 
\begin{align}
m_{\varepsilon_{k}} & \le I_{\varepsilon_{k}}(u_{k})=\left(\frac{1}{2}-\frac{1}{p}\right)\|u_{k}\|_{\varepsilon_{k}}^{2}+\omega\left(\frac{1}{4}-\frac{1}{p}\right)G_{\varepsilon_{k}}(u_{k})\label{eq:mepsk}\\
 & =\frac{1}{4}\|u_{k}\|_{\varepsilon_{k}}^{2}+\left(\frac{1}{4}-\frac{1}{p}\right)|u_{k}^{+}|_{p,\varepsilon_{k}}^{p}\le m_{\varepsilon_{k}}+2\delta_{k}\le m_{\infty}+3\delta_{k}\nonumber 
\end{align}
 for $k$ large enough (see Remark \ref{rem:limsup}), and, for any
$q\in\Omega$, 
\[
\int_{B(q,r/2)\cap\Omega}\Gamma(u_{k})\le\left(1-\eta\right)m_{\infty}.
\]
By Ekeland principle and by definition of ${\mathcal N}_{\varepsilon_{k}}$
we can assume 
\begin{equation}
\left|I'_{\varepsilon_{k}}(u_{k})[\varphi]\right|\le\sigma_{k}\|\varphi\|_{\varepsilon_{k}}\text{ where }\sigma_{k}\rightarrow0.\label{eq:ps}
\end{equation}
By Lemma \ref{lem:gamma} there exists a set $P_{k}^{\varepsilon_{k}}\in{\mathcal P}_{\varepsilon_{k}}$
such that 
\[
\frac{1}{\varepsilon_{k}^{3}}\int_{P_{k}^{\varepsilon_{k}}}|u_{k}^{+}|^{p}dx\ge\gamma,
\]
morevoer, by Lemma \ref{lem:crucial} we have that $\frac{d(P_{k}^{\varepsilon_{k}},\partial\Omega)}{\varepsilon_{k}}\rightarrow+\infty$. 

We choose a point $q_{k}\in P_{k}^{\varepsilon_{k}}$ and we define
the set 
\[
\Omega_{\varepsilon_{k}}:=\frac{1}{\varepsilon_{k}}\left(\Omega-q_{k}\right)=\left\{ z\in\mathbb{R}^{3}\ :\ \varepsilon_{k}z+q_{k}\in\Omega\right\} .
\]
We remark that, since $\Omega\supset B(q_{z},d(P_{k}^{\varepsilon_{k}},\partial\Omega))$
and since $\frac{d(P_{k}^{\varepsilon_{k}},\partial\Omega)}{\varepsilon_{k}}\rightarrow+\infty$,
we have $\Omega_{\varepsilon_{k}}\nearrow\mathbb{R}^{3}$. We define,
the function $w_{k}:\mathbb{R}^{3}\rightarrow\mathbb{R}$ as 
\[
w_{k}(z)=\left\{ \begin{array}{ccc}
u_{k}(\varepsilon_{k}z+q_{k}) &  & z\in\Omega_{\varepsilon_{k}}\\
0 &  & \text{elsewhere}
\end{array}\right..
\]
We have that $w_{k}\in H^{1}(\mathbb{R}^{3})$. By equation (\ref{eq:mepsk})
we have 
\[
\|w_{k}\|_{H^{1}(\mathbb{R}^{3})}^{2}=\|u_{k}\|_{\varepsilon_{k}}^{2}\le C.
\]
So $w_{k}\rightarrow w$ weakly in $H^{1}(\mathbb{R}^{3})$ and strongly
in $L_{\text{loc}}^{s}(\mathbb{R}^{3})$ for $2\le s<6$. 

Let $\psi_{k}(x):=\psi_{\varepsilon_{k}}(u_{k})(x)$ where $\psi_{\varepsilon_{k}}(u_{k})$
solves $-\varepsilon_{k}^{2}\Delta v=qu_{k}^{2}\text{ in }\Omega$,
with Dirichlet boundary condition, and again define $\tilde{\psi}_{k}:\mathbb{R}^{3}\rightarrow\mathbb{R}$
as 
\[
\tilde{\psi}_{k}(z)=\left\{ \begin{array}{ccc}
\psi_{k}(\varepsilon_{k}z+q_{k}) &  & z\in\Omega_{\varepsilon_{k}}\\
0 &  & \text{elsewhere}
\end{array}\right..
\]
 and, as in the proof of Lemma \ref{lem:stimaGeps} and by (\ref{eq:mepsk})
we have $\|\tilde{\psi}_{k}\|_{D^{1,2}(\mathbb{R}^{3})}\le C$. So
there exists $\bar{\psi}\in D^{1,2}(\mathbb{R}^{3})$ such that $\tilde{\psi}_{k}\rightharpoonup\bar{\psi}$
weakly in $D^{1,2}(\mathbb{R}^{3})$ and in $L^{6}(\mathbb{R}^{3})$.
We show that $\bar{\psi}=\psi_{\infty}(w)$, $\psi_{\infty}(w)$ being
the solution of $-\Delta v=qw^{2}\text{ in }\mathbb{R}^{3}$. 

Take a function $f\in C_{0}^{\infty}(\mathbb{R}^{3})$. We have $\text{supp}(f)\subset B(0,T)\subset\Omega_{\varepsilon_{k}}$,
for some $T>0$, and for $k$ sufficently large. We define 
\[
f_{k}(x)=f\left(\frac{x-q_{k}}{\varepsilon_{k}}\right)
\]
so $\text{supp}(f_{k})\subset B(q_{k},\varepsilon_{k}T)\subset\Omega$.
Since $\psi_{k}=\psi_{\varepsilon_{k}}(u_{k})$ is a weak solution
of (\ref{eq:psi-eps}), we have $\varepsilon_{k}^{2}\int_{\Omega}\nabla\psi_{k}\nabla f_{k}dx=q\int_{\Omega}u_{k}^{2}f_{k}dx$.
Now, with the change of variables $x=\varepsilon_{k}z+q_{k}$ we have
\[
\frac{1}{\varepsilon_{k}}\int_{\Omega}\nabla\psi_{k}\nabla f_{k}dx=\int_{\Omega_{\varepsilon_{k}}}\nabla\tilde{\psi}_{k}\nabla fdz=\int_{\mathbb{R}^{3}}\nabla\tilde{\psi}_{k}\nabla fdz\rightarrow\int_{\mathbb{R}^{3}}\nabla\bar{\psi}\nabla fdz.
\]
In a similar way 
\[
\frac{1}{\varepsilon_{k}^{3}}\int_{\Omega}u_{k}^{2}f_{k}dx=\int_{\Omega_{\varepsilon_{k}}}w_{k}^{2}fdz=\int_{\mathbb{R}^{3}}w_{k}^{2}fdz\rightarrow\int_{\mathbb{R}^{3}}w^{2}fdz,
\]
hence for any $f\in C_{0}^{\infty}(\mathbb{R}^{3})$ 
\[
\int_{\mathbb{R}^{3}}\nabla\bar{\psi}\nabla fdz=q\int_{\mathbb{R}^{3}}w^{2}fdz
\]
and we have proved that $\bar{\psi}=\psi_{\infty}(w)$, as claimed. 

Moreover, since $\|f_{k}\|_{\varepsilon_{k}}=\|f\|_{H^{1}(\mathbb{R}^{3})}$
and by (\ref{eq:ps}), we have $\left|I'_{\varepsilon_{k}}(u_{k})[f_{k}]\right|\rightarrow0$
as $k\rightarrow\infty$. Also, by the change of variables $x=\varepsilon_{k}z+q_{k}$
we get
\begin{align*}
I'_{\varepsilon_{k}}(u_{k})[f_{k}]= & \frac{1}{\varepsilon_{k}^{3}}\int_{\Omega}\varepsilon_{k}^{2}\nabla u_{k}\nabla f_{k}+u_{k}f_{k}+\omega qu_{k}\psi_{k}f_{k}-(u_{k}^{+})^{p-1}f_{k}d\mu_{g}\\
= & \int_{\Omega_{\varepsilon}}\nabla w_{k}\nabla f+w_{k}f+\omega qw_{k}\tilde{\psi}_{k}f-(w_{k}^{+})^{p-1}fdz\\
= & \int_{\mathbb{R}^{3}}\nabla w_{k}\nabla f+w_{k}f+\omega qw_{k}\tilde{\psi}_{k}f-(w_{k}^{+})^{p-1}fdz\\
\rightarrow & \int_{\mathbb{R}^{3}}\nabla w\nabla f+wf+\omega qw\psi_{\infty}(w)f-(w_{k}^{+})^{p-1}fdz=I'_{\infty}(w)[f]
\end{align*}
since $w_{k}\rightarrow w$ weakly in $H^{1}(\mathbb{R}^{3})$ and
$\tilde{\psi}_{k}\rightharpoonup\psi_{\infty}(w)$ weakly in $D^{1,2}(\mathbb{R}^{3})$
and in $L^{6}(\mathbb{R}^{3})$. So we get that $w$ is a weak solution
of the limit problem (\ref{eq:PL}). By Lemma \ref{lem:gamma} and
by the choice of $q_{k}$ we have that $w\ne0$, so $w>0$, $w\in\mathcal{N}_{\infty}$,
and $I_{\infty}(w)\ge m_{\infty}$. 

By weak convergence of $w_{k}$, by the defintion of $\mathcal{N}_{\infty}$
and by (\ref{eq:mepsk}) we get 
\begin{align*}
m_{\infty} & \le I_{\infty}(w)=\frac{1}{4}\|w\|_{H^{1}}^{2}+\left(\frac{1}{4}-\frac{1}{p}\right)|w^{+}|_{p}^{p}\\
 & \le\liminf_{k\rightarrow\infty}\left(\frac{1}{4}\|w_{k}\|_{H^{1}}^{2}+\left(\frac{1}{4}-\frac{1}{p}\right)|w_{k}^{+}|_{p}^{p}\right)\le\liminf_{k\rightarrow\infty}\left(m_{\infty}+3\delta_{k}\right)=m_{\infty}
\end{align*}
so we have that $w_{k}\rightarrow w$ strongly in $H^{1}(\mathbb{R}^{3})$
and in $L^{p}(\mathbb{R}^{3})$ and that $w$ is a ground state for
the limit problem (\ref{eq:PL}).

Given $T>0$, by the definition of $w_{k}$ we get, for $k$ large
enough
\begin{multline}
\int_{B(0,T)}\left[\frac{1}{4}|\nabla w_{k}|^{2}+\frac{1}{4}|w_{k}|^{2}+\left(\frac{1}{4}-\frac{1}{p}\right)|w_{k}^{+}|^{p}\right]dz\\
=\frac{1}{\varepsilon^{3}}\int_{B(q_{k},\varepsilon_{k}T)}\frac{1}{4\varepsilon_{k}}|\nabla u_{k}|^{2}+\frac{1}{4\varepsilon_{k}^{3}}|u_{k}|^{2}+\frac{1}{\varepsilon_{k}^{3}}\left(\frac{1}{4}-\frac{1}{p}\right)|u_{k}^{+}|^{p}dx\\
=\int_{B(q_{k},\varepsilon_{k}T)}\Gamma(u_{k})dx\le\int_{B(q_{k},r/2)\cap\Omega}\Gamma(u_{k})dx\le\left(1-\eta\right)m_{\infty},\label{eq:contr}
\end{multline}
we remark here that, eventually, $B(q_{k},\varepsilon_{k}T)\subset\Omega$,
since $\frac{d(q_{k},\partial\Omega)}{\varepsilon_{k}}\ge\frac{d(P_{k}^{\varepsilon_{k}},\partial\Omega)}{\varepsilon_{k}}\rightarrow\infty$
by Lemma \ref{lem:crucial}. 

On the other hand, $w_{k}\rightarrow w$ and $\tilde{\psi}_{k}\rightarrow\psi_{\infty}(w)$
in $L^{2}(B(0,T))$ for any $T>0$ and for $2\le s<6$. Thus, since
$m_{\infty}=I_{\infty}(w)=\left(\frac{1}{2}-\frac{1}{p}\right)|w^{+}|^{p}-\frac{\omega}{4}G(w)$,
for any $\eta$ it is possible to choose $T$ such that
\[
\int_{B(0,T)}\left[\frac{1}{4}|\nabla w_{k}|^{2}+\frac{1}{4}|w_{k}|^{2}+\left(\frac{1}{4}-\frac{1}{p}\right)|w_{k}^{+}|^{p}\right]dz>(1-\eta)m_{\infty},
\]
which contradict (\ref{eq:contr}), so the lemma is proved for $u\in{\mathcal N}_{\varepsilon}\cap I_{\varepsilon}^{m_{\varepsilon}+2\delta}$.

The above arguments also prove that 
\[
\liminf_{k\rightarrow\infty}m_{\varepsilon_{k}}\ge\lim_{k\rightarrow\infty}I_{\varepsilon_{k}}(u_{k})=m_{\infty}.
\]
and, in light of (\ref{eq:limsup}), this leads to 
\begin{equation}
\lim_{\varepsilon\rightarrow0}m_{\varepsilon}=m_{\infty}.\label{eq:mepsminfty}
\end{equation}
Hence, when $\varepsilon,\delta$ are small enough, ${\mathcal N}_{\varepsilon}\cap I_{\varepsilon}^{m_{\infty}+\delta}\subset{\mathcal N}_{\varepsilon}\cap I_{\varepsilon}^{m_{\varepsilon}+2\delta}$
and the general claim follows.
\end{proof}
\begin{prop}
There exists $\delta_{0}\in(0,m_{\infty})$ such that for any $\delta\in(0,\delta_{0})$
and any $\varepsilon\in(0,\varepsilon(\delta_{0})$ (see Proposition
\ref{prop:phieps}), for every function $u\in{\mathcal N}_{\varepsilon}\cap I_{\varepsilon}^{m_{\infty}+\delta}$
it holds $\beta(u)\in\Omega^{+}$. Moreover the composition 
\[
\beta\circ\Phi_{\varepsilon}:\Omega^{-}\rightarrow\Omega^{+}
\]
 is s homotopic to the immersion $i:\Omega^{-}\rightarrow\Omega^{+}$ 
\end{prop}
\begin{proof}
By Proposition \ref{prop:conc}, for any function $u\in{\mathcal N}_{\varepsilon}\cap I_{\varepsilon}^{m_{\infty}+\delta}$,
for any $\eta\in(0,1)$ and for $\varepsilon,\delta$ small enough,
we can find a point $q=q(u)\in\Omega$ such that 
\[
\int_{B(q,r/2)\cap\Omega}\Gamma(u)>\left(1-\eta\right)m_{\infty}.
\]
Moreover, since $u\in{\mathcal N}_{\varepsilon}\cap I_{\varepsilon}^{m_{\infty}+\delta}$
we have 
\[
I_{\varepsilon}(u)=\int_{\Omega}\Gamma(u)\le m_{\infty}+\delta.
\]
Hence, since $\Gamma(u)\ge0$,
\begin{eqnarray*}
|\beta(u)-q| & \le & \frac{\left|\int_{\Omega}(x-q)\Gamma(u)\right|}{\int_{\Omega}\Gamma(u)}\\
 & \le & \frac{\left|\int_{B(q,r/2)}(x-q)\Gamma(u)\right|}{\int_{\Omega}\Gamma(u)}+\frac{\left|\int_{\Omega\smallsetminus B(q,r/2)}(x-q)\Gamma(u)\right|}{\int_{\Omega}\Gamma(u)}\\
 & \le & \frac{r}{2}+2\text{\ diam}(\Omega)\left[\frac{\delta+\eta m_{\infty}}{m_{\infty}+\delta}\right],
\end{eqnarray*}
and the second term can be made arbitrarily small, choosing $\delta,\eta$
sufficiently small. The second claim of the theorem is standard.
\end{proof}

\section{The set $T_{\varepsilon}$\label{sec:Teps}}

In this section we construct a contractible set in the space $H_{0}^{1}(\Omega)$.
This will prove the existence of another solution with higher energy.

Let $V\in C_{0}^{\infty}(\mathbb{R}^{3})$, $V\ge0$ a non identically
zero function. Take a point $q_{0}\in\Omega^{-}$ and define
\[
v_{\varepsilon}(x)=V\left(\frac{x-q_{0}}{\varepsilon}\right).
\]
 Since $V$ is compactly supported, $v_{\varepsilon}\in H_{0}^{1}(\Omega)$
eventually in $\varepsilon$. We define the set of functions 
\[
C_{\varepsilon}:=\left\{ u(x)=\theta v_{\varepsilon}+(1-\theta)W_{q,\varepsilon}\text{ for }q\in\overline{\Omega^{-}}\ ,\theta\in[0,1]\right\} ,
\]
where $W_{q,\varepsilon}$ is defined as in (\ref{eq:Weps}). 

We have that $C_{\varepsilon}$ is a compact, contractible set in
$H_{0}^{1}(\Omega)$. Now we define
\[
T_{\varepsilon}:=\left\{ t_{\varepsilon}(u)u\ :\ u\in C_{\varepsilon}\right\} 
\]
where $t_{\varepsilon}(u)$ is the unique positive value such that
$t_{\varepsilon}(u)u\in\mathcal{N}_{\varepsilon}$ as in Lemma \ref{lem:nehari}.
Since $t_{\varepsilon}(u)$ is a continuous function, we have also
that $T_{\varepsilon}$ is a compact contractible set in $\mathcal{N}_{\varepsilon}$.
Also, we point out that every function in $T_{\varepsilon}$ is positive
by definition. We define
\[
c_{\varepsilon}:=\max_{u\in T_{\varepsilon}}I_{\varepsilon}(u).
\]

\begin{lem}
\label{lem:C}There exists $C\in\mathbb{R}$ such that $c_{\varepsilon}\le C$
for $\varepsilon$ sufficiently small.
\end{lem}
\begin{proof}
Since $\theta\in[0,1],$ by rescaling and by Remark \ref{w} we have
that 
\[
\|\theta v_{\varepsilon}+(1-\theta)W_{q,\varepsilon}\|_{\varepsilon}\le\|v_{\varepsilon}\|_{\varepsilon}+\|W_{q,\varepsilon}\|_{\varepsilon}\rightarrow\|V\|_{H_{0}^{1}(\mathbb{R}^{3})}+\|U\|_{H_{0}^{1}(\mathbb{R}^{3})}
\]
so 
\[
\|\theta v_{\varepsilon}+(1-\theta)W_{q,\varepsilon}\|_{\varepsilon}\le2\left(\|V\|_{H_{0}^{1}(\mathbb{R}^{3})}+\|U\|_{H_{0}^{1}(\mathbb{R}^{3})}\right),
\]
 and in the same way 
\[
|\theta v_{\varepsilon}+(1-\theta)W_{q,\varepsilon}|_{\varepsilon,p}\le2\left(\|V\|_{L^{p}(\mathbb{R}^{3})}+\|U\|_{L^{p}(\mathbb{R}^{3})}\right).
\]
Moreover, since $v_{\varepsilon}\ge0$ and $W_{q,\varepsilon}\ge0$
we have
\[
|\theta v_{\varepsilon}+(1-\theta)W_{q,\varepsilon}|_{\varepsilon,p}\ge\max\left\{ \theta|v_{\varepsilon}|_{\varepsilon,p},(1-\theta)|W_{q,\varepsilon}|_{\varepsilon,p}\right\} \ge\frac{1}{2}\min\left\{ |v_{\varepsilon}|_{\varepsilon,p},|W_{q,\varepsilon}|_{\varepsilon,p}\right\} 
\]
and, by Remark \ref{w}
\[
|\theta v_{\varepsilon}+(1-\theta)W_{q,\varepsilon}|_{\varepsilon,p}\ge\frac{1}{4}\min\left\{ \|V\|_{L^{p}(\mathbb{R}^{3})},\|U\|_{L^{p}(\mathbb{R}^{3})}\right\} .
\]
Similarly 
\[
\|\theta v_{\varepsilon}+(1-\theta)W_{q,\varepsilon}\|_{\varepsilon}\ge\frac{1}{4}\min\left\{ \|V\|_{L^{2}(\mathbb{R}^{3})},\|U\|_{L^{2}(\mathbb{R}^{3})}\right\} .
\]
Finally, arguing as in Lemma \ref{lem:stimaGeps} we have 
\begin{align*}
\left|G_{\varepsilon}(\theta v_{\varepsilon}+(1-\theta)W_{q,\varepsilon})\right| & =\left|\frac{1}{\varepsilon^{3}}\int_{\Omega}\left(\theta v_{\varepsilon}+(1-\theta)W_{q,\varepsilon}\right)^{2}\psi_{\varepsilon}dx\right|\\
 & \le|\theta v_{\varepsilon}+(1-\theta)W_{q,\varepsilon}|_{\varepsilon,12/5}^{2}|\psi_{\varepsilon}|_{\varepsilon,6}
\end{align*}
where $\psi_{\varepsilon}$ is the solution of $-\varepsilon^{2}\Delta\psi_{\varepsilon}=q(\theta v_{\varepsilon}+(1-\theta)W_{q,\varepsilon})$.
Moreover 
\[
\frac{1}{\varepsilon}\int_{\Omega}|\nabla\psi_{\varepsilon}|^{2}=\frac{q}{\varepsilon^{3}}\int(\theta v_{\varepsilon}+(1-\theta)W_{q,\varepsilon})\le|\theta v_{\varepsilon}+(1-\theta)W_{q,\varepsilon}|_{\varepsilon,12/5}^{2}|\psi_{\varepsilon}|_{\varepsilon,6}
\]
and, since there exists a constant $C$ which does not depend on $\varepsilon$
such that $|\psi_{\varepsilon}|_{\varepsilon,6}\le C\left(\frac{1}{\varepsilon}\int_{\Omega}|\nabla\psi_{\varepsilon}|^{2}\right)^{1/2},$
we get 
\[
\left|G_{\varepsilon}(\theta v_{\varepsilon}+(1-\theta)W_{q,\varepsilon})\right|\le C|\theta v_{\varepsilon}+(1-\theta)W_{q,\varepsilon}|_{\varepsilon,\frac{12}{5}}^{4}\le C(\|V\|_{L^{\frac{12}{5}}(\mathbb{R}^{3})}^{4}+\|U\|_{L^{\frac{12}{5}}(\mathbb{R}^{3})}^{4})
\]
uniformly in $\varepsilon$. Now, given $u\in C^{\varepsilon}$ we
have (see Lemma \ref{lem:nehari}) that $t_{\varepsilon}(u)$ is the
unique positive solution of 
\[
\|u\|_{\varepsilon}^{2}+t^{2}\omega G_{\varepsilon}(u)-t^{p-2}|u|_{\varepsilon,p}=0,
\]
 and by the above estimates we conlclude that there exists two constants
$c_{1},c_{2}>0$ independent on $\varepsilon$ and $u\in C_{\varepsilon}$
such that $c_{1}\le t_{\varepsilon}(u)\le c_{2}$. At this point,
for all $t_{\varepsilon}(u)u\in T_{\varepsilon}$ we have 
\[
I_{\varepsilon}(t_{\varepsilon}(u)u)=\frac{t_{\varepsilon}^{2}(u)}{4}\|u\|_{\varepsilon}^{2}+\left(\frac{1}{4}-\frac{1}{p}\right)t_{\varepsilon}(u)^{p}|u^{+}|_{p,\varepsilon}^{p}\le C
\]
 for some constant $C,$ and the proof follows.
\end{proof}

\end{document}